\documentclass[12pt,reqno]{amsart}
\usepackage{amssymb,mathrsfs,amsmath,amsfonts,amsthm,graphicx,pdfpages}
\usepackage{xcolor}
\usepackage{tikz, tikz-cd}
\usepackage[section]{algorithm}
\usepackage{algorithmic}

\newcommand{\define}[1]{{\bf \boldmath{#1}}}

\newcommand{\namedset}[1]{\mathbb{#1}}

\newcommand{\R}{\namedset R}

\newcommand{\namedcat}[1]{\mathsf{#1}}
\newcommand{\Cat}{\namedcat{Cat}}

\newcommand{\Set}{\namedcat{Set}}

\newcommand{\Mor}{\mathrm{Mor}\,}
\newcommand{\Ob}{\mathrm{Ob}\,}
\newcommand{\Vect}{\mathsf{Vect}}
\newcommand{\sdiamond}{\scalebox{0.6}{$\,\diamond\,$}}

\newcommand{\Ch}{\mathrm{Ch}_{\bullet}}

\newcommand{\sVect}{\mathsf{sVect}}
\makeatletter
\newcommand*{\relrelbarsep}{.386ex}
\newcommand*{\relrelbar}{%
  \mathrel{%
    \mathpalette\@relrelbar\relrelbarsep
  }%
}
\newcommand*{\@relrelbar}[2]{%
  \raise#2\hbox to 0pt{$\m@th#1\relbar$\hss}%
  \lower#2\hbox{$\m@th#1\relbar$}%
}
\providecommand*{\rightrightarrowsfill@}{%
  \arrowfill@\relrelbar\relrelbar\rightrightarrows
}
\providecommand*{\leftleftarrowsfill@}{%
  \arrowfill@\leftleftarrows\relrelbar\relrelbar
}
\providecommand*{\xrightrightarrows}[2][]{%
  \ext@arrow 0359\rightrightarrowsfill@{#1}{#2}%
}
\providecommand*{\xleftleftarrows}[2][]{%
  \ext@arrow 3095\leftleftarrowsfill@{#1}{#2}%
}
\makeatother

\usepackage{subfiles}
\usepackage{stackengine}
\usepackage{mathtools}
\usepackage{color}
\usepackage[utf8]{inputenc}
\usepackage{csquotes}
\usepackage[a4paper,top=3cm,bottom=3cm,inner=3cm,outer=3cm]{geometry}

\usepackage{comment}

\usepackage{graphicx}
\usepackage{adjustbox}
\usepackage[all,2cell]{xy}\UseAllTwocells\SilentMatrices
\definecolor{darkgreen}{rgb}{0,0.45,0}
\usepackage[colorlinks,citecolor=darkgreen,urlcolor=gray,final,hyperindex,linktoc=page,pagebackref]{hyperref}

\usepackage[capitalize]{cleveref}
\crefname{equation}{}{}
\crefname{item}{}{}
\usepackage{enumerate}

\newtheorem*{thm*}{Theorem}
\theoremstyle{remark}
\newtheorem*{rmk*}{Remark}
\newtheorem*{lem*}{Lemma}
\theoremstyle{definition}
\newtheorem*{defn*}{Definition}
\newtheorem*{cor*}{Corollary}
\theoremstyle{definition}
\newtheorem*{examples*}{Examples}
\newtheorem{prop*}{Proposition}

\theoremstyle{plain}
\newtheorem{thm}{Theorem}[section]
\theoremstyle{plain}
\newtheorem{prop}[thm]{Proposition}
\theoremstyle{remark}

\theoremstyle{plain}

\theoremstyle{plain}

\theoremstyle{definition}
\newtheorem{defn}[thm]{Definition}
\theoremstyle{definition}

\newtheorem{hyp}[thm]{Hypothesis}
\theoremstyle{plain}

\newcommand{\maps}{\colon}

\definecolor{purple(x11)}{rgb}{0.8, 0, 0.8}

\usetikzlibrary{arrows,positioning,fit,matrix,shapes.geometric,external}

\newcommand*\pgfdeclareanchoralias[3]{%
  \expandafter\def\csname pgf@anchor@#1@#3\expandafter\endcsname
     \expandafter{\csname pgf@anchor@#1@#2\endcsname}}

\tikzset{
    circnode/.style={
      circle, draw=red, very thin, outer sep=0.025em, minimum size=2em,
      fill=red, text centered},
    integral/.style={
      circle, draw=black, very thick, outer sep=0.025em,
      minimum size=2em, fill=blue!5, text centered},
    multiply/.style={
      circle, draw=black, very thick, outer sep=0.025em,
      minimum size=2em, fill=blue!5, text centered},
    zero/.style={
      circle, draw=black, very thick, minimum size=0.15cm, fill=black,
      inner sep=0, outer sep=0},
    bang/.style={
      circle, draw=black, very thick, minimum size=0.15cm, fill=green!10,
      inner sep=0, outer sep=0},
    delta/.style={
      regular polygon, regular polygon sides=3, minimum size=0.4cm, inner
      sep=0, outer sep=0.025em, draw=black, very thick, fill=green!10},
    codelta/.style={
      regular polygon, regular polygon sides=3, shape border rotate=180, minimum size=0.4cm,
      inner sep=0, outer sep=0.025em, draw=black, very thick, fill=green!10},
    plus/.style={
      regular polygon, regular polygon sides=3, shape border rotate=180, minimum size=0.4cm,
      inner sep = 0, outer sep=0.025em, draw=black, very thick, fill=black},
    coplus/.style={
      regular polygon, regular polygon sides=3, minimum size=0.4cm,
      inner sep = 0, outer sep=0.025em, draw=black, very thick, fill=black},
    sqnode/.style={
      regular polygon, regular polygon sides=4, minimum size=2.6em,
      draw=black, very thick, inner sep=0.2em, outer sep=0.025em,
      fill=yellow!10, text centered},
    bigcirc/.style={
      circle, draw=black, very thick, text width=1.6em, outer sep=0.025em,
      minimum height=1.6em, fill=blue!5, text centered}
}

\tikzstyle{tri}=[regular polygon,regular polygon sides=3,shape border rotate=1
80,fill=none,draw=black]

\pgfdeclareanchoralias{regular polygon}{corner 2}{left copy}
\pgfdeclareanchoralias{regular polygon}{corner 3}{right copy}
\pgfdeclareanchoralias{regular polygon}{corner 3}{addend}
\pgfdeclareanchoralias{regular polygon}{corner 2}{summand}

\usetikzlibrary{positioning}
\definecolor {processblue}{cmyk}{0.9,0.5,0,0}

\tikzstyle{simple}=[-,line width=2.000]
\tikzstyle{arrow}=[-,postaction={decorate},decoration={markings,mark=at position .5 with {\arrow{>}}},line width=1.100]
\pgfdeclarelayer{edgelayer}
\pgfdeclarelayer{nodelayer}
\pgfdeclarelayer{bg}
\pgfsetlayers{bg,edgelayer,nodelayer,main}

\tikzstyle{none}=[inner sep=-1pt]
\tikzstyle{species}=[circle,fill=none,draw=black,scale=0.75]
\tikzstyle{transition}=[rectangle,fill=none,draw=black,scale=1.15]
\tikzstyle{empty}=[circle,fill=none, draw=none]
\tikzstyle{inputdot}=[circle,fill=black,draw=black, scale=.5]
\tikzstyle{dot}=[circle,fill=black,draw=black]
\tikzstyle{bounding}=[circle,dashed, fill=none,draw=black, scale=9.00]
\tikzstyle{triplebounding}=[circle,dashed, fill=none,draw=black, scale=30.00]
\tikzstyle{simple}=[-,draw=black,line width=1.000]
\tikzstyle{inarrow}=[-,draw=black,postaction={decorate},decoration={markings,mark=at position .5 with {\arrow{>}}},line width=1.000]
\tikzstyle{tick}=[-,draw=black,postaction={decorate},decoration={markings,mark=at position .5 with {\draw (0,-0.1) -- (0,0.1);}},line width=1.000]
\tikzstyle{inputarrow}=[->,draw=black, shorten >=.05cm]

\UseTwocells
\usetikzlibrary{backgrounds,circuits,circuits.ee.IEC,shapes,fit,matrix}

\tikzstyle{tri}=[regular polygon,regular polygon sides=3,shape border rotate=1
80,fill=none,draw=black]

\tikzstyle{simple}=[-,line width=2.000]
\tikzstyle{arrow}=[-,postaction={decorate},decoration={markings,mark=at position .5 with {\arrow{>}}},line width=1.100]
\pgfdeclarelayer{edgelayer}
\pgfdeclarelayer{nodelayer}
\pgfsetlayers{edgelayer,nodelayer,main}

\tikzstyle{none}=[inner sep=-1pt]

\definecolor{lblue}{rgb}{0,250,255}

\tikzstyle{species}=[circle,fill=yellow,draw=black,scale=1]
\tikzstyle{transition}=[rectangle,fill=lblue,draw=black,scale=1]
\tikzstyle{morphism}=[rectangle,fill=pink,draw=black,scale=1]

\definecolor{darkgreen}{rgb}{0,0.45,0}
\usepackage[colorlinks,citecolor=darkgreen,urlcolor=gray,final,hyperindex,linktoc=page,pagebackref]{hyperref}

\title{Why is Homology so Powerful?}

\author{Jade Master}

\begin{document}
\maketitle
\begin{abstract}
My short answer to this question is that homology is powerful because it computes invariants of higher categories. In this article we show how this true by taking a leisurely tour of the connection between category theory and homological algebra.
\textbf{Dependencies}: This article assumes familiarity with the basics of category theory and the basics of algebraic topology. 
\end{abstract}
\section{Extending Eckmann-Hilton}
There are many reasons why homology is powerful and this article gives only one perspective. Most of my explanation boils down to the Eckmann-Hilton argument. This is the following theorem:

\begin{thm}
Let $X$ be a set with two binary unital operations
\[ + \maps X \times X \to X \text{ and } \circ \maps X \times X \to X \]
such that $\circ$ is a homomorphism of $+$, i.e. 
\[ (a + b) \circ ( c + d) = a \circ c + b \circ d.\]
Then $+$ and $\circ$ are the same operation and this operation is commutative.
\end{thm}

The proof of this argument is a lot of fun.

\begin{proof}
Let $1$ denote the unit for $\circ$ and $0$ denote the unit for $+$. First we will show that $1=0$. This follows from the chain of equations
\begin{align*}
f & = f+0 \\
 &= (1 \circ f) + (0 \circ 1) \\
 & = (1 + 0) \circ (f +1) \\
 &= 1 \circ (f+ 1) \\
 & = (f+1). \\
\end{align*}
Therefore $0$ and $1$ are both units for the operation $+$. Because units must be unique, we have that $1=0$. Now we show that the two operations are the same
\begin{align*}
f \circ g &= (f+0) \circ (0 +g) \\
&= (f+1) \circ (1 +g) \\
&= (f \circ 1) + (1 \circ g) \\
&= f +g.
\end{align*}
Lastly we show that this operation is commutative
\begin{align*}
    f \circ g &= (0+f) \circ (g +0) \\
    &= (0 \circ g) + (f \circ 0)\\
    &= (1 \circ g) + (f \circ 1) \\
    &= g + f \\
    &= g \circ f.
\end{align*}

\end{proof}
The relationship between the two operations can be thought of as a clock, where the arrows represent equalities. The following image from Wikimedia commons shows this with the two operations represented by $\oplus$ and $\otimes$ \cite{commons}.
\begin{center}
\includegraphics[scale=0.4]{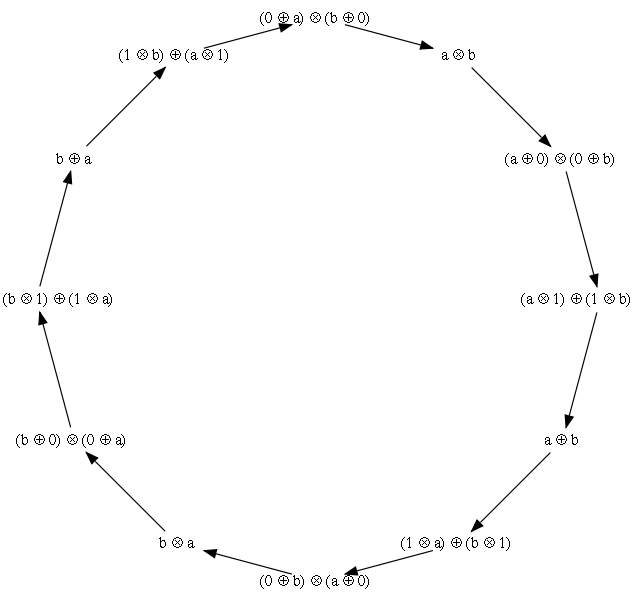}
\end{center}
One way of thinking of a category is as a monoid whose operation is partial. For this reason, the Eckmann-Hilton argument bears on categories equipped with a binary operation. One way to equip categories with operations like this is through internalization.

\begin{defn}
Let $V$ be a category with finite pullbacks. A category $C$ internal to $V$ is a graph in $V$
\[
\begin{tikzcd}
\Mor C \ar[r,shift left=.5ex,"s"] \ar[r,shift right=.5ex,"t",swap] & \Ob C 
\end{tikzcd}
\]
along with an identity assigning morphism
\[i  \maps \Ob C \to \Mor  C \]
and composition morphism
\[\circ \maps \Mor C \times_{\Ob C} \Mor  C \to \Mor   C\]
commuting suitably with the source and target maps. These morphisms are required to satisfy the axioms of unitality and associativity expressed as commutative diagrams.
\end{defn}
 Let $\Vect$ be the category where objects are vector spaces over the real numbers and morphisms are linear transformations. Categories internal to $\Vect$ were first studied by Baez and Crans in \emph{Higher Dimensional Algebra VI: Lie 2-Algebras} \cite{baezcrans}. Here we explicitly describe what these internal categories are like.
 \begin{defn}
 A category $\mathbf{C}$ internal to $\Vect$ is a 
 \begin{itemize}
     \item a vector space of objects $C_0$,
     \item a vector space of morphisms $C_1$,
     \item source and target linear transformations $s,t \maps C_1 \to C_0$,
     \item an identity assigning linear transformation $i \maps C_0 \to C_1$ and,
     \item a composition linear transformation $\circ \maps C_1 \times_{C_0} C_1 \to C_1$
 \end{itemize}
 satisfying the required axioms. A category internal to $\Vect$ is called a \define{2-vector space}.
 \end{defn}
An example of a $2$-vector space is as follows: Let $\mathbf{C}$ be a 2-vector space with $C_1=\mathbb{R}^4$ and $C_0 = \mathbb{R}^2$. Let the source function be given by
\begin{align*}
    s \maps & \mathbb{R}^4 \longrightarrow \mathbb{R}^2 \\
     &(a,b,c,d) \mapsto (a,b)
\end{align*}
let the target function be given by
\begin{align*}
    t \maps & \mathbb{R}^4 \longrightarrow \mathbb{R}^2 \\
     & (a,b,c,d) \mapsto (a+c,b+d),
\end{align*}
and let the identity assigning map be given by
\begin{align*}
    i \maps & \mathbb{R}^2 \longrightarrow \mathbb{R}^4 \\
     & (a,b) \mapsto (a,b,0,0).
\end{align*}
A morphism $f = (a,b,c,d)$ in $\mathbf{C}$ can be drawn in the plane as follows:
\begin{center}
 \begin{tikzpicture}
 \draw[step=.5cm,gray,very thin] (-2,-2) grid (2,2);
 \draw[thick,->] (-2,-1) -- (0,.5);
 \draw(-2,-1) -- (-2,-1) node[anchor=north] {$(a,b)$};
 \draw(0,.5) -- (0,.5) node[anchor=south] {$(a+c,b+d)$};
 \end{tikzpicture}.
 \end{center}
 The first two components of $f$ give the origin of the arrow and the second two components give the arrow itself. The geometry of the situation suggests a natural categorical structure i.e. a composition rule satisfying the axioms of a category. This categorical structure can be defined without any reference to the specific choice of source, target and identity of $\mathbf{C}$. Thus, we now describe this structure for an arbitrary reflexive graph internal to $\Vect$ although it may be helpful to keep $\mathbf{C}$ in the back of your mind.

 For a morphism $f \maps x \to y$ ($s(f)=x$ and $t(f) = y$), we define the \define{arrow part} of $f$ by
 \[\hat{f}= f -i(s(f))= f- i(x). \]
 The idea is that this``translates $f$ to $0$". Now the source of $\hat{f}$ is $0$
 \[s(\hat{f}) = s(f - i(x))= s(f) - s(i(x)) = x-x = 0\]
 because the source and target of $i(x)$ are both given by $x$.
The target of $\hat{f}$ is now given by
 \[ t(f-i(s(f))) = t(f) -t(i(s(f)))= t(f) - s(f) = y-x\]
 Note that every $f$ can be written as
 \[ f = f - i(x) + i(x) = \hat{f} + i(x)\] in a unique way.
 This allows us to think of $f \maps x \to y$ in $C$ is as the vector $\hat{f}$ in the plane pointing from $i(x)$ to $i(y)$. The arrow part of $i(x)$ and $i(y)$ are $0$ because the the source and target maps commute with the identity map. This justifies the lack of an arrow attached to $i(x)$ and $i(y)$ in the above picture. Given a pair of morphisms $f \maps x \to y$ and $g \maps y \to z$,
 \begin{center}
 \begin{tikzpicture}
 \draw[step=.5cm,gray,very thin] (-2,-2) grid (2,2);
 \draw[thick,->] (-2,-1) -- (0,.5) node[midway,above] {$f$};
 \draw[thick,->] (0,.5) -- (2,-.5) node[midway,above] {$g$};
 \draw(-2,-1) -- (-2,-1) node[anchor=north] {$i(x)$};
 \draw(0,.5) -- (0,.5) node[anchor=south] {$i(y)$};
 \draw(2,-.5) -- (2,-.5) node[anchor=north] {$i(z)$};
 \end{tikzpicture}
 \end{center}
 we can compose to get a morphism $g \sdiamond f \maps x \to z$
  \begin{center}
 \begin{tikzpicture}
 \draw[step=.5cm,gray,very thin] (-2,-2) grid (2,2);
 \draw[thick,->] (-2,-1) -- (0,.5) node[midway,above] {$f$};
 \draw[thick,->] (0,.5) -- (2,-.5) node[midway,above] {$g$};
 \draw[thick,->] (-2,-1) -- (2,-.5) node [midway,below] {$g \sdiamond f$};
 \draw(-2,-1) -- (-2,-1) node[anchor=north] {$i(x)$};
 \draw(0,.5) -- (0,.5) node[anchor=south] {$i(y)$};
 \draw(2,-.5) -- (2,-.5) node[anchor=north] {$i(z)$};
 \end{tikzpicture}
 \end{center}
 Formally, the composite $g \sdiamond f$ is given by 
 \[g \sdiamond f = \hat{f} + \hat{g} +i(s(f)). \]
 The arrow part of $g \sdiamond f$ is now given by the sum of the arrow parts of $f$ and of $g$. The source of $g \sdiamond f$ is
 \begin{align*} 
 s(g \sdiamond f) &= s(\hat{f} +\hat{g} + i(x))\\
 &= 0 + 0 + s(i(x))\\
 &= x
 \end{align*}
 and the target of $g \sdiamond f$ is given by
 \begin{align*}
 t(g \circ f) &= t(\hat{f} +\hat{g} +t(i(s(f)) \\
 &= t(\hat{f})+ t(\hat{g}) + s(f) \\
 &= y - x + z - y + x \\
 &=z.
 \end{align*} 
 The last step of this computation requires commutativity of vector sum. If the sum was not commutative, then the above composition would not form the structure of a category on the underlying reflexive graph of $\mathbf{C}$. $i$ does assign elements to their identity morphism under $\sdiamond$. Composing on the right gives
\begin{align*}
 f \sdiamond i(x) & = 0 + \hat{f} +i(s(i(x)))\\
 &= (f - i(s(f))) + i(x) \\
 &= f.
 \end{align*}
Similarly, composing on the left results in
\begin{align*}
    i(y) \sdiamond f &= \hat{i(y)} + \hat{f} + i(x) \\
    &= i(y) - i(s(i(y)) + \hat{f} + i(x)\\
    &= i(y) - i (y)  + \hat{f} + i(x) \\
    &= \hat{f} + i(x) \\
    &= f - i(x) + i(x) \\
    &= f.
\end{align*}
Therefore the composition rule $\sdiamond$ defines a composition rule on $\mathbf{C}$. Because we made no reference to a specific choice of source, target and identity map, composition given by $\sdiamond$ defines the structure of a category on \emph{any} reflexive graph internal to $\Vect$. What's really surprising is that this is the only way to define composition in a $2$-vector space.
 \begin{prop}\label{comp}
 Every $2$-vector space has composition defined as above.
 \end{prop}
 \begin{proof}
Let $C$ be a 2-vector space
\[\begin{tikzcd} C_1 \ar[r,shift left=1.1ex,"s"] \ar[r,shift right=1.1ex,"t",swap] & \ar[l,"i"description] C_0 \end{tikzcd}\]
with composition map
\[ \circ \maps C1 \times_{C0} C1 \to C1.\]As shown above, the underlying reflexive graph of $C$ can be turned into a category via the rule
\[ g \sdiamond f = \hat{f} + \hat{g} + i (s(f)) \]
where $\,\hat{}\,$ denotes the \emph{arrow part} of a morphism. Because $\circ$ is a linear transformation it satisfies the law
 \[ (g \circ f) + 
 (g' \circ f') = (g  + g') \circ 
 (f + f')\]
 where $f \maps x \to y$, $g \maps y \to z$, $f' \maps x' \to y'$ and $g' \maps y' \to z'$. This equation is called the \define{interchange law}. The interchange law allows us to apply a version of the Eckmann-Hilton argument to the operations $\sdiamond$ and $\circ$. Indeed for $f \maps x \to y$ and $g \maps y \to z$ we have that
 \begin{align*}
     g \circ f + 1_y &= g \circ f + 1_y \circ 1_y \\
     &= (g+1_y) \circ (f +1_y)\\
\end{align*}
     by the interchange law. Using commutativity and the interchange law again we get that 
\begin{align*}
     (g+1_y) \circ (f +1_y) &= (1_y + g) \circ (f + 1_y) \\
     &= f \circ 1_y + 1_y \circ g \\
     &= f+g 
\end{align*}
However, we can decompose this using the arrow parts of $f$ and $g$ to get the $\sdiamond$ composition:
\begin{align*}
     f+g &= 1_x + \hat{f} + 1_y + \hat{g} \\
     &= (1_x + \hat{f} + \hat{g}) + 1_y \\
     &= g \sdiamond f + 1_y
\end{align*}
Setting $y=0$ gives that $1_y=0$ as well because $i$ is a linear transformation. Therefore, when $y=0$, the above sequence of equations gives that $g \circ f = g \sdiamond f$.
\end{proof}
\noindent Proposition \ref{comp} extends to the following equivalence of categories:
 \begin{prop}\label{vectgraph}
 There is an equivalence of categories 
 \[\namedcat{2}\text{-}\mathsf{Vect} \cong \mathrm{RGraph}(\Vect) \]
 where $\mathrm{RGraph}(\Vect)$ is the category of reflexive graphs internal to $\Vect$.
 \end{prop}
 
 \begin{proof}
 the left inverse sends reflexive graphs to the category with composition rule given by $\sdiamond$ and sends morphisms of reflexive graphs to the unique functor which respects this composition rule. The right inverse sends $\Vect$-categories to their underlying reflexive graph and $\Vect$-functors to their underlying morphisms of reflexive graphs. A detailed proof of this proposition can be found in Crans' thesis \cite{crans}.
 \end{proof}

 \section{From Categories to Chain Complexes}
 When I first learned about chain-complexes I didn't understand what sort of thing they were trying to describe. What I was looking for was some down-to-earth explanation of their motivation. I got a clue about this when I learned the definition of a homotopy between chain maps.
\begin{defn}
Given chain maps $f,g \maps C_\cdot \to D_\cdot$,
a \define{homotopy} $\alpha_\cdot \maps f \Rightarrow g$ is a family of functions $\alpha_n$ of the following form:
\[ \begin{tikzcd}
\ldots \ar[r] &C_{n+1} \ar[d,"f_{n+1} -g_{n+1}",swap] \ar[r,"\delta_{n+1}"] & C_n \ar[d,"f_n - g_n"description,swap] \ar[dl,"\alpha_n",swap] \ar[r,"\delta_n"] & C_{n-1} \ar[dl,"\alpha_{n-1}"] \ar[d,"f_{n-1} - g_{n-1}"] \ar[r] & \ldots \\
\ldots \ar[r] & D_{n+1} \ar[r,"\delta_{n+1}",swap]& D_n \ar[r,"\delta_n",swap] & D_{n-1} \ar[r] & \ldots 
\end{tikzcd}\]
However, the above triangles do not commute. Instead they satisfy the equations
\[f_n - g_n = \delta_{n+1} \circ \alpha_n + \alpha_{n-1} \circ \delta_n.\]
\end{defn}
I noticed that this definition reminded me of the definition of a natural transformation.
\begin{defn}
Given functors $F,G \maps C \to D$ a \define{natural transformation} $\alpha \maps F \Rightarrow G$, is a function of the form:
\[ \begin{tikzcd}
\Mor C \ar[d,shift left=.5ex,"F_1"] \ar[d,shift right=.5ex,"G_1",swap] \ar[r,shift left=.5ex,"s"] \ar[r,shift right=.5ex,"t",swap] & \Ob C \ar[dl,"\alpha"] \ar[d,shift left=.5ex,"F_0"] \ar[d,shift right=.5ex,"G_o",swap] \\
\Mor D \ar[r,shift left=.5ex,"s"] \ar[r,shift right=.5ex,"t",swap] & \Ob D
\end{tikzcd}\]
Not all the triangles here commute, but we do have the equations
\begin{align*}
 s \circ \alpha & = F_0 \\
 t \circ \alpha & = G_0 
\end{align*}
expressing that $\alpha$ offers comparison morphisms between the images of $F$ and $G$. $\alpha$ must also satisfy a naturality condition expressing compatibility with composition in $C$ and $D$.
\end{defn}

 Both of these definitions consist of maps going diagonally and up a dimension and it turns out that the equations they must satisfy are related as well. This relationship is part of a larger story which relates higher categories to chain complexes in order to reason about them more effectively. To see how this works, we need to understand three things:
\begin{enumerate}
\item How categories can be turned into simplicial sets,
\item How simplical sets can be turned into simplicial vector spaces and,
\item How simplicial vector spaces can be turned into chain complexes.
\end{enumerate}
Once we understand these three things, we will have a 2-functor 
\[ \Ch \maps \Cat \to \Ch(\Vect)\]
which connects the disparate worlds of category theory and homological algebra. These three things will be addressed by the proceeding three subsections.
\subsection{The Nerve Construction}
Categories can be thought of as simplicial sets which only have interesting information in dimensions $0$,$1$, and $2$. The nerve construction makes this precise by providing a full and faithful embedding from $\Cat$ to $\namedcat{sSet}$. 
\begin{defn}\label{nerve} For a category $C$, its nerve is a simplicial set 
\[ N(C) \maps \Delta^{op} \to \Set \]
with 
\[N(C)[n]= \Cat( [n] , C)\]
i.e. the set of functors from the poset $[n]$ to the category $C$ \cite{nLab}. In other words, this is the set of composable n-chains of morphisms in $C$. The boundary map $d_i \maps N
(C)[n+1] \to N(C)[n]$ comes in two cases: 
\begin{itemize}
    \item If $i=0$ or $n$ then it sends an $n$-chain to an $n-1$ chain which forgets the first and the last morphism in the chain respectively.
    \item Otherwise, $d_i$ acts by composing $i$-th morphism with the $i+1$-th morphism to get an $n$-chain.
\end{itemize}
The degeneracy maps 
\[ s_i \maps N(c)[n] \to N(c)[n
+1] \]
turn $n$-chains into $n+1$-chains by inserting and identity in the $i$-th spot.
\end{defn}
At first this definition seems too intuitive to be the right thing. John Baez said this about the nerve:
\begin{displayquote}
    When I first heard of this idea I cracked up. It seemed like an insane sort of joke. Turning a category into a kind of geometrical object built of simplices? What nerve! What use could this possibly be? \cite{twf}
\end{displayquote}
Category theory is a field of math where you not only guess answers to questions but also the questions and definitions. With enough experience, certain definitions in category theory will feel inevitable, as if they are the only natural way to define things. The following definition feels that way:

\begin{defn}
For a functor between categories $F \maps C \to D$, there is a natural transformation 
\[ N(F) \maps N(C) \Rightarrow N(D) \]
defined on $0$-cells by the object component of $F$. For higher dimensional simplices, the map 
\[ N(F)[n] \maps  N(C)[n] \to N(D)[n]\]
sends a commuting $n$-chain
\[
\begin{tikzcd}
x_0 \ar[r,"f_0"] & x_1 \ar[r,"f_1"] & \ldots \ar[r,"f_{n-1}"] & x_{n-1} \ar[r,"f_n"] & x_n
\end{tikzcd}
\]
to its image under $F$
\[\begin{tikzcd}
F(x_0) \ar[r,"F(f_0)"] & F(x_1) \ar[r,"F(f_1)"] & \ldots \ar[r,"F(f_{n-1})"] & F(x_{n-1}) \ar[r,"F(f_n)"] & F(x_n)
\end{tikzcd}. \]
\end{defn}
 \noindent $N(F)$ is well defined because every functor sends commuting diagrams to commuting diagrams. 
 
 It's a surprising and incredible fact that only the $0$, $1$ and $2$-chains contain all the necessary data of your category. Let $\Delta^{op}_{\leq n}$ be the full subcategory of $\Delta^{op}$ which only contains the objects $[0],[1],\ldots, [n]$. Then the inclusion
\[ \Delta^{op}_{\leq n} \hookrightarrow \Delta^{op}\]
induces a truncation functor by precomposition
\[\mathsf{trunc}_{\leq n} \maps \Set^{\Delta^{op}} \to \Set^{\Delta_{\leq n}^{op}}.\]
This functor sends a simplicial set $X \maps \Delta^{op} \to \Set$ to it's trunctation $X_{\leq n} \maps \Delta_{\leq n}^{op} \to \Set$
which forgets about the the $k$-simplices for $k>n$. It turns out that truncation has both a left and right adjoint. The property of adjointness can be used to give a slick definition:
\begin{defn}
Let 
\[\mathsf{Sk}_n \maps \Set^{\Delta_{\leq n}^{op}} \to  \Set^{\Delta^{op}} \]
be the left adjoint to $n$-truncation and let 
\[\mathsf{Cosk}_n \maps \Set^{\Delta_{\leq n}^{op}} \to  \Set^{\Delta^{op}}\]
be the right adjoint to $n$-truncation. A simplicial set is called \define{$n$-skeletal} if it is isomorphic to something in the image of $\mathsf{Sk}_n$ and called \define{$n$-coskeletal} if it is isomorphic to something in the image of $\mathsf{Cosk}_n$.
\end{defn}
The skeleton freely adds degenerate simplices i.e. the only simplices for $k>n$ are the degenerate simplices of lower dimension. The coskeleton freely fills in higher dimensional simplices when possible. This means that whenever a set of simplices with dimension less than $n$ outline a simplex of dimension greater than $n$ then that higher dimensional simplex is freely included in the coskeleton. In either case, the $k$-simplices for $k>n$ are trivial in the sense that they are completely determined by simplices in dimension less than $n$. Nerves of categories form simplicial sets which are $2$-coskeletal.
\begin{thm}\label{embed}
 The nerve construction 
 \[ N \maps \Cat \hookrightarrow \namedcat{sSet}\]
 is a full and faithful functor whose image contains only $2$-coskeletal simplical sets.
\end{thm}\noindent To understand this theorem, it will be useful to unpack the definition of the nerve. For a category $C$, the $0$-simplices of $N(C)$ are given by
\[ N(C)[0] = \Ob C\]
and the $1$-simplices of $N(C)$ are given by
\[N(C)[1] = \Mor C.\]
The 2-cells are more interesting. They can be thought of as commuting triangles 
\[
\begin{tikzcd}
& x_1 \ar[ddr,"g"] &  \\
& & \\
x_0 \ar[uur,"f"] \ar[rr,"h",swap] & & x_2
\end{tikzcd}
\]
These encode the relations between morphisms. Just like how groups and other algebraic gadgets can be described using generators and relations, categories can be described with two things
\begin{itemize}
    \item its data i.e. objects and morphisms and,
    \item its relations, i.e. equations between morphisms and their compositions.
\end{itemize}
This statement is justified by the fact that $\Cat$ is the category of algebras for the ``free-category on a directed graph" monad. In particular, this means that every category can be described as a graph homomorphism
\[ A \maps L(X) \to X \]
where $L(X)$ is the underlying graph of the free category on some graph $X$. Here, the map $A$ picks out relations between arbitrary compositions of morphisms in a category whose underlying graph is given by $X$. So, because the $0$,$1$, and $2$-simplices of $N(C)$ contain the objects, morphisms, and relations of $C$, it makes at least intuitive sense that these simplices capture all the essential information of $C$. A $3$-simplex of $N(C)$ is a commuting square
\[
\begin{tikzcd}
x_0 \ar[r,"f"] \ar[d,"k",swap] & x_1 \ar[d,"g"]\\
x_3  & \ar[l,"h"] x_2
\end{tikzcd}
\]
However, instead we could write this as two commuting triangles
\[
\begin{tikzcd}
& x_1 \ar[ddr,"h"] &  \\
& & \\
x_0 \ar[uur,"g \circ f"] \ar[rr,"k",swap] & & x_2
\end{tikzcd}
\quad
\begin{tikzcd}
& x_1 \ar[ddr,"h \circ g"] &  \\
& & \\
x_0 \ar[uur,"f"] \ar[rr,"k",swap] & & x_2.
\end{tikzcd}
\]
Actually, these two triangles are the two inner boundaries of the above $3$-simplex according to the nerve construction. In this way, every $3$-simplex is redundant because the data it represents is already contained as $2$-simplices. Note that this phenomenon already occurs with groups. It is the fact that every product of three elements $x_1x_2x_3$ can be turned into two different products of two elements, $(x_1x_2)x_3$ and $x_1(x_2x_3)$, by adding parentheses. This is why groups only have a binary operation rather than an $n$-ary operation for every natural number $n$.

\subsection{Simplicial Vector Spaces}
Simplicial vector spaces are just like simplicial sets, except that the $n$-simplices form a vector space rather than a set.
\begin{defn}
 A \define{simplicial vector space} is a functor
 \[\Delta^{op} \to \Vect \]
\end{defn}\noindent To turn a simplicial set $\Delta^{op} \to \Set $ into a simplicial vector space,
we will compose it with a sensible functor $\mathbf{F} \maps \Set \to \Vect$. The functor $\mathbf{F}$ is as follows:
\begin{prop}
 Let 
 \[ \mathbf{U} \maps \Vect \to \Set\]
 be the forgetful functor which sends every vector space to its underlying set. Then $U$ has a left adjoint 
 \[\mathbf{F}\maps \Set \to \Vect\]called the \define{free vector space functor}. A set $X$ is sent to the set $\mathbb{R}^X$ of functions from $X$ to $\mathbb{R}$ which are nonzero on only a finite subset of 
 $X$. For a function $f \maps X \to Y$, 
 \[ \mathbf{F}(f) \maps \R^X \to \R^Y\]
 is the unique linear transformation which extends $f$.
\end{prop}
Roughly, $\mathbf{F}$ is a reasonable functor to choose because we want it to preserve the information in each simplicial set as faithfully as possible. For a set $X$, $\mathbf{F}(X)$ is a vector space which
\begin{itemize}
    \item includes the elements of $X$ and,
    \item only includes other elements if they are necessary to make $\mathbf{F}(X)$ into a vector space. These include all sums and scalar multiples of elements in $X$ without any relations.
\end{itemize}
This is perfect because we're not doing anything too fancy. To summarize:
\begin{defn}
 There is a functor 
 \[ (\Delta^{op})^\mathbf{F} \maps \namedcat{sSet} \to \namedcat{sVect} \]
 which composes every simplicial set with the free vector space on a set functor. For a natural transformation of simplicial sets $\alpha \maps X \to Y$, this functor whiskers the natural transformation with the functor $\mathbf{F}$.
\end{defn}
\subsection{The Dold-Kan Correspondence}
To complete our quest of turning categories into chain complexes, we have to turn simplical sets into chain complexes. This is done by taking an alternating sum of the boundary maps.
\begin{defn}\label{alternating}
 Given a simplicial vector space $X\maps \Delta^{op} \to \namedcat{Vect}$, the \define{alternating face map chain complex} of $X$ is chain complex
 \[\begin{tikzcd} \ldots \ar[r,"\delta_{n+1}"] & X([n]) \ar[r,"\delta_n"] & X([n-1]) \ar[r,"\delta_{n-1}"] & \ldots \end{tikzcd}\]
 The boundary maps are defined by
 \[ \delta_n := \sum_{i=1}^{n} (-1)^i d_i\]
 where the $d_i \maps X([n]) \to X([n-1])$ are the face maps of $X$.
\end{defn}\noindent This gives a functor
\[A \maps \sVect \to \Ch(\Vect)\]
in a natural way. For a natural transformation $\alpha \maps X \to Y$ between simplicial vector spaces, there is a chain map
\[A(\alpha) \maps A(X) \to A(Y) \]
whose $n$-th component is given by $\alpha_{X[n]}$.
Usually people don't stop here. The alternating face map chain complex can be \define{normalized} by quotienting each vector space of $n$-chains by the subspace of degenerate simplices. The composition of normalization and taking alternating face map chain complex is called the Dold-Kan correspondence. It is famous because it forms an equivalence of categories between simplical vector spaces and chain complexes.

\subsection{Natural Transformations to Homotopies}
Now let's put all this together. Given a category
\[\begin{tikzcd}C = C_1 \ar[r,shift left=.5ex,"s"] \ar[r,shift right=.5ex,"t",swap] & C_0 \end{tikzcd} \]
we take it's nerve and free simplicial vector space. This results in a simplicial vector space $(\Delta^{op})^\mathbf{F} \circ N (C) \maps \Delta^{op} \to \namedcat{Vect}$ with 
\[(\Delta^{op})^\mathbf{F} \circ N(C)([0]) = \mathbf{F}(C_0) : = \mathbf{C_0},\] \[ (\Delta^{op})^\mathbf{F} \circ N (C) ([1]) = \mathbf{F}(C_1) : = \mathbf{C_1},
\]
\[(\Delta^{op})^\mathbf{F} \circ N (C) ([2])= \mathbf{F} ( N (C)[2]) : = \mathbf{N (C)[2]} \]
\[\vdots \]
where the application of $\mathbf{F}$ to a set or a function is indicated by making its symbol bold.
 $(\Delta^{op})^\mathbf{F} \circ N(C)$ can be turned into a chain complex
\[\begin{tikzcd} \ldots \ar[r] & \mathbf{C_1} \ar[r,"\mathbf{t}-\mathbf{s}"] & \mathbf{C_0} \ar[r,"0"] & 0 \ar[r] & \ldots\end{tikzcd} \]
 Although we didn't say it, everything here is 2-functorial, i.e. it defines a 2-functor
\[\mathrm{Ch} \maps \Cat \to \mathrm{Ch}_{\cdot} (\Vect). \]
This means that for a natural transformation
\[ \begin{tikzcd}
C_1 \ar[d,shift left=.5ex,"F_1"] \ar[d,shift right=.5ex,"G_1",swap] \ar[r,shift left=.5ex,"s"] \ar[r,shift right=.5ex,"t",swap] &  C_0 \ar[dl,"\alpha"] \ar[d,shift left=.5ex,"F_0"] \ar[d,shift right=.5ex,"G_0",swap] \\
D_1 \ar[r,shift left=.5ex,"s"] \ar[r,shift right=.5ex,"t",swap] &  D_0
\end{tikzcd}\]
we get a homotopy between chain maps as follows:
\[ \begin{tikzcd}
\ar[r] \ldots & \mathbf{N(C) [2]}\ar[r,"\delta_2^C"] \ar[d,"\mathbf{N(G)_{
[2]}}-\mathbf{N(F)_{[2]}}"description] & \mathbf{C_1} \ar[dl,"\alpha_2",swap] \ar[d,"\mathbf{G_1}-\mathbf{F_1}"description] \ar[r,"\mathbf{t}-\mathbf{s}"] & \mathbf{C_0} \ar[d,"\mathbf{G_0}-\mathbf{F_0}"description,swap] \ar[dl,"\alpha_1",swap] \ar[r,"0"] & 0 \ar[dl,"0"] \ar[d,"0"] \ar[r] & \ldots \\
\ar[r] \ldots & \mathbf{N(D)[2]}\ar[r,"\delta_2^D",swap] & \mathbf{D_1} \ar[r,"\mathbf{t'}-\mathbf{s'}",swap]& \mathbf{D_0} \ar[r,"0",swap] & 0 \ar[r] & \ldots 
\end{tikzcd}\]
\noindent $\alpha_1$ is defined to be the natural transformation $\mathbf{F}(\alpha) := \boldsymbol{\alpha}$, the linear extension of the natural transformation $\alpha$. Recall that because $\alpha$ is a natural transformation, it satisfies the equations
\[ s \circ \alpha = F_0 \text{ and } t \circ \alpha = G_0\]
Therefore, 
\begin{align*}
(\mathbf{t'}-\mathbf{s'}) \circ \boldsymbol{\alpha} + 0 \circ 0 &= (\mathbf{t'}-\mathbf{s'}
) \circ \boldsymbol{\alpha} \\
&= \mathbf{t'} \circ \boldsymbol{\alpha} -
 \mathbf{s'} \circ \boldsymbol{\alpha} \\
 &= \mathbf{G_0} - \mathbf{F_0}
\end{align*}
so the two squares on the right do indeed satisfy the equations for a homotopy of chain maps. This fact reflects that natural transformations have the right source and target.

The homotopy condition for the two squares on the left expresses the fact that natural transformations are compatible with morphisms in $C$. For a morphism $f \maps x \to y$ in $C$, the naturality square is a $3$-simplex in $N(D)$ 
\[
\begin{tikzcd}
F(x)\ar[d,"\alpha_x",swap] \ar[r,"F(f)"] & F(y)\ar[d,"\alpha_y"] \\
G(x) \ar[r,"G(f)",swap] & G(y)
\end{tikzcd}
\]
$\alpha_2$ must send $f$ to a sum of triangular $2$-simplices in $D$. Luckily,the above square has two nice triangles as boundaries given by collapsing edges with composition. These are
\[
A=
\begin{tikzcd}
F(x)\ar[dr,"G(f) \circ \alpha_x",swap] \ar[r,"F(f)"] & F(y)\ar[d,"\alpha_y"] \\
 & G(y)
\end{tikzcd}
\]
and
\[
B=\begin{tikzcd}
F(x)\ar[d,"\alpha_x",swap] \ar[dr,"\alpha_y \circ F(f)"]  & \\
G(x) \ar[r,"G(f)",swap] & G(y)
\end{tikzcd}
\]
To avoid choosing one triangle we take the second one and subtract the first. Indeed the map
\[\alpha_2 \maps \mathbf{C_1} \to \mathbf{N(D)[2]} \]
sends a morphism $f \maps x \to y$ in $C$ to the difference $B-A$ and is freely extended to the rest of $\mathbf{C_1}$. This map satisfies the chain homotopy condition.

\begin{prop}
For a morphism $f \maps x \to y$ in $C$, there is an equation
\[\delta^D_2 \circ \alpha_2 (f) + \alpha_1  \circ \delta^C_1(f) = G(f)- F(f)\] so that $\alpha_1$ and $\alpha_2$ form components of a chain homotopy.
\end{prop}

\begin{proof} 
According to definitions \ref{nerve} and \ref{alternating} we have that 
\[\delta^C_1 =  \mathbf{t} - \mathbf{s}\]
and that $\delta_2^D$ is the linear extension of the mapping
\begin{center}
\[ \begin{tikzcd} & x_1 \ar[ddr,"g"] & & &\\
 & & &\mapsto &f - g\circ f  + g \\
x_0 \ar[uur,"f"] \ar[rr,"g \circ f",swap] &  &x_2& & \end{tikzcd} \]
\end{center}
This gives that
\begin{align*} 
    \delta_2^D \circ \alpha_2(f) + \alpha_1 \circ \delta_1^C (f) & = \delta_2^D \circ \alpha_2 (f)+ \alpha_1( \mathbf{t} - \mathbf{s}) (f) \\
    &=\delta_2^D \circ \alpha_2 (f)+ \alpha_1 
    (y)- \alpha_1 (x)
    &= \delta_2^D( B- A) + \alpha_y - \alpha_x.
\end{align*}
Using the definition of $\delta_2^D$,
\begin{align*} &= ( \alpha_x - \alpha_y \circ F(f) + G(f) ) -( F(f) - G(f) \circ \alpha_x + \alpha_y)+ \alpha_y - \alpha_x \\
&= G(f) - F(f) +  G(f) \circ \alpha_x - \alpha_y \circ F(f) 
\end{align*}
However, because the naturality square for $f$ commutes we have that 
\[G(f) \circ \alpha_x = \alpha_y \circ F(f) \]
so that\[ G(f) \circ \alpha_x - \alpha_y \circ F(f) = 0.\]
Applying this to the above equation gives that 
\[ G(f) - F(f) +  G(f) \circ \alpha_x - \alpha_y \circ F(f) = G(f) - F(f). \]

\end{proof}





\section{Going Backwards}
 So far we have sketched a way of interpreting categories as chain complexes using the $2$-functor
\[\Ch = A \circ (\Delta^{op})^\mathbf{F} \circ N \maps \Cat \to \Ch(\Vect) \]Mathematicians are very suggestible. A way of interpreting $A$'s as $B$'s suggests that $B$'s can be thought of as a generalization of the $A$'s. With this logic, chain complexes generalize categories. In what way do they do this? Theorem \ref{embed} shows that nerves of categories contain only interesting data in their $0$, $1$ and $2$-simplices. Because the $2$-functors $A$ and $(\Delta^{op})^\mathbf{F}$ don't produce interesting data where there wasn't any before, the chain complexes which come from the nerves of categories have the same property. Therefore chain complexes generalize categories by containing higher dimensional content.

There is another notion of category with "higher dimensional content" called n-categories for which a good exposition can be found in \cite{baezn}. $n$-categories are notoriously complex. Todd Trimble drafted a 51-page definition of a weak $4$-category \cite{trimblet}

\begin{quote}
  In 1995, at Ross Street's request, I gave a very explicit description of weak 4-categories, or tetracategories as I called them then, in terms of nuts-and-bolts pasting diagrams, taking advantage of methods I was trying to develop then into a working definition of weak n-category. Over the years various people have expressed interest in seeing what these diagrams look like -- for a while they achieved a certain notoriety among the few people who have actually laid eyes on them (Ross Street and John Power may still have copies of my diagrams, and on occasion have pulled them out for visitors to look at, mostly for entertainment I think).
\end{quote}
This quote is referring to weak $n$-categories. Strict $n$-categories have much simpler axioms as every composition operation is associative strictly. Regardless, classifying and understanding $n$-categories is a large and arduous mathematical quest which is relevant to many subjects in math. For example $n$-categories can be used to do rewriting theory in a more sophisticated way \cite{coherence}. 

It is a theorem of the heart that Proposition \ref{vectgraph} extends as follows. Maybe it has been proved somewhere but I do not know where.
\begin{hyp}
There is a suitable equivalence
\[\mathsf{nGraph}(\Vect) \cong \mathsf{nCat}(\Vect) \]
between $n$-dimensional graphs internal to $\Vect$ and $n$-categories internal to $\Vect$.
\end{hyp}
An $n$-dimensional graph should be something like a graph with edges between edges, and edges between those edges continuing until you get $n$-levels deep. What this equivalence would say is that every $n$-dimensional internal to $\Vect$ already has an intricate network of interacting composition operations built into it in a unique way. Homology is so powerful because it allows you to reason about these complicated networks of composition just by thinking about vector spaces and linear maps. 

\section{Acknowledgments}
Thank you to my advisor John Baez for helping me when I got stuck. Thank you to the users on reddit who pointed out many errors. Thank you to Olivia Borghi, Aaron Goodwin, Connor Malin, David Jaz Meyers, and David Weisbart for helping me revise this article. Thank you to Christian Williams for helping me figure out how natural transformations turn into chain homotopies. Thank you to everyone around me who supported me while writing this, both emotionally and physically. Thank you to anyone I forgot, math is created by communities and not individuals. This work was produced on Tongva land.

\bibliographystyle{acm}
\bibliography{references}

\end{document}